\documentclass[10pt]{amsart}
\usepackage{bm}

\usepackage{amsmath}
\usepackage{amsthm}
\usepackage{amssymb}

\newcommand{\R}{\mathbb{R}}
\newcommand{\C}{\mathbb{C}}
\newcommand{\Z}{\mathbb{Z}}
\newcommand{\N}{\mathbb{N}}
\newcommand{\lr}[1]{\langle #1 \rangle}
\newcommand{\eps}{\varepsilon}
\newcommand{\norm}[1]{\| #1 \|}

\newcommand{\wh}[1]{\widehat{#1}}
\newcommand{\ol}[1]{\overline{#1}}

\newcommand{\dx}{\partial_x}

\newcommand{\M}{M^0_{2,1}}

\newtheorem{thm}{Theorem}[section]

\newtheorem{dfn}[thm]{Definition}
\newtheorem{lem}[thm]{Lemma}

\theoremstyle{remark}

\allowdisplaybreaks[1]

\usepackage{graphicx}
\usepackage{color}

\title[Sharp ill-posedness of DKG in 1d]{Sharp ill-posedness of the Dirac-Klein-Gordon system in one dimension}

\author[S. Machihara]{Shuji Machihara}
\address{Department of Mathematics, Faculty of Science, Saitama University, 255 Shimo-Okubo, Sakura-ku, Saitama City 338-8570, Japan}
\email{machihara@rimath.saitama-u.ac.jp}

\author[M. Okamoto]{Mamoru Okamoto}
\address{Division of Mathematics and Physics, Faculty of Engineering, Shinshu University, 4-17-1 Wakasato, Nagano City 380-8553, Japan}
\email{m\_okamoto@shinshu-u.ac.jp}



\date{\today}

\allowdisplaybreaks[1]

\numberwithin{equation}{section}

\begin{document}

\begin{abstract}
We show the ill-posedness of the Cauchy problem for the Dirac-Klein-Gordon 
system in one dimension in the critical Sobolev space. From this, 
we finish the classification of the regularities for which this problem 
is well-posed or ill-posed. 
\end{abstract}

\maketitle

\section{Introduction}

We consider the Cauchy problem for the Dirac-Klein-Gordon system:
\begin{equation}\label{DKG}
\left\{
\begin{aligned}
& (i\gamma_0\partial_{t}+\gamma_1\dx)\psi+m\psi=\phi \psi, \\
& (\partial_t ^2-\dx ^2 + M^2) \phi = \psi ^{\ast} \gamma ^0 \psi, \\
& \psi (0,x) = \psi _0 (x) , \  \phi (0,x) = \phi _0(x), 
\ \partial _t \phi (0,x) = \phi _1(x),
\end{aligned}
\right.
\end{equation}
where $\psi = \bigl( \begin{smallmatrix} \psi _1 \\ \psi _2 \end{smallmatrix} \bigr): 
\R ^{1+1} \rightarrow \C ^2$ and $\phi : \R ^{1+1} \rightarrow \R$ are unknown functions 
of $(t,x)\in \R ^{1+1}$, 
$\psi _0 = \bigl( \begin{smallmatrix} \psi _{0,1} \\ \psi _{0,2} \end{smallmatrix} \bigr)
: \R \rightarrow \C ^2$ and $\phi _0 , \phi _1:\R \rightarrow \R$ are given functions
of $x\in\R$.
Here,
$m$ and $M$ are nonnegative constants, and $\gamma_0, \gamma_1$ are $2\times2$ Hermitian 
matrices 
\begin{equation}\label{D-matrix}
\gamma_0 = \begin{pmatrix} 1 & 0 \\ 0 & -1 \end{pmatrix}, \quad
\gamma _1 = \begin{pmatrix} 0 & -i \\ i & 0 \end{pmatrix}
\end{equation}
which satisfy the anticommutation relations and leads to 
$(i\gamma_0\partial_{t}+\gamma_1\dx)^2=(-\partial_t ^2+\dx ^2)I_2$
where $I_2$ is the $2\times2$ identity matrix, 
$\psi^{\ast}$ denotes the conjugate transpose of $\psi$. 

We will discuss the well-posedness for the Cauchy problem in \eqref{DKG}. 
A problem is called well-posed if a solution exists uniquely and the solution map is continuous. 
The last property is important in our main theorem in this paper. 
If the solution map is continuous, 
the sequence of initial data $u_n(0)\to u(0)$ requires the convergence of 
the corresponding sequence of solutions 
$u_n(t)\to u(t)$ with $t>0$. 
Here, we are concerned with initial data in the Sobolev spaces $H^{s} (\R)$.  
For $s\in\R$, the Sobolev norm associated with 
regularity $s$ is given by
\begin{align*}
\|f\|_{H^s}=\|\lr{\xi}^s\hat{f}\|_{L^2}=\left(\int_{\R}\lr{\xi}^{2s}|\hat{f}(\xi)|^2d\xi\right)^{1/2}
\end{align*}
where $\hat{f}(\xi)=\int_{\R}e^{-ix\xi}f(x)dx$ is the Fourier transform
of $f(x)$. 
We consider the well-posedness in $H^s (\R)\times H^r(\R)$ which means 
\begin{align*}
(\psi,\phi)\in H^s(\R)\times H^r(\R)
\end{align*}
where we use the symbols $s$ and $r$ for the regularities of $\psi$ and $\phi$ 
respectively. 
For brevity, we shall refer to well-posedness from 
initial data $(\psi_0,\phi_0,\phi_1)\in H^s(\R)\times H^r(\R)\times H^{r-1}(\R)$ 
to the solution  $(\psi,\phi,\partial_t\phi)\in H^s(\R)\times H^r(\R)\times H^{r-1}(\R)$
as well-posedness in $H^s(\R)\times H^r(\R)$.

\subsection{Known results}
Since we are interested in the classification of 
well-posedness in this paper, 
we restrict ourselves to consider the time local issue, namely  
time local well-posedness or not, and 
we will not go into the time global issue. 
If any of the conditions which stipulate well-posedness, 
namely existence, uniqueness or continuous dependence 
on the initial data fails, 
we say that the problem is ill-posed. 
The first author with Nakanishi and Tsugawa in \cite{MNT10} proved 
time local well-posedness of \eqref{DKG} 
in the region $|s|\le r\le s+1$ except the following forbidden point 
\begin{align}\label{criticalpoint}
(s,r)=\left(-\frac12,\frac12\right).
\end{align}
We call this point {\it the critical point} in this paper.  
In the same paper \cite{MNT10}, they proved ill-posedness in the two 
regions max$\{0,r\}<s$ 
and max$\{\frac12,s+1\}<r$. The authors of the current paper proved 
ill-posedness in the region $s<0,r<\frac12, s+r<0$ in \cite{MacOka16}, 
and on the two lines $s=0, r<0$ and $s<-\frac12,r=\frac12$ 
in \cite{MacOka162}. 
Thus, well-posedness or ill-posedness of the problem remained open 
only at the critical point \eqref{criticalpoint} 
(see Figure \ref{graph1}). 

Other earlier papers which obtained 
the well-posedness in subsets of the region $|s|\le r\le s+1$ are 
\cite{bournaveas1, b-g1,b-g2,chadam,CG, fang0, fang1, FangHuang, 
machihara3, pecher, s-t1, s-t2}. There are also global well-posedness 
results with $s<0$ in which they used Bourgain's frequency decomposition 
technique or I-method with a help of the charge conservation law 
\cite{candy2, selberg1, T}. 

\begin{figure}[h]
\includegraphics[width=7cm]{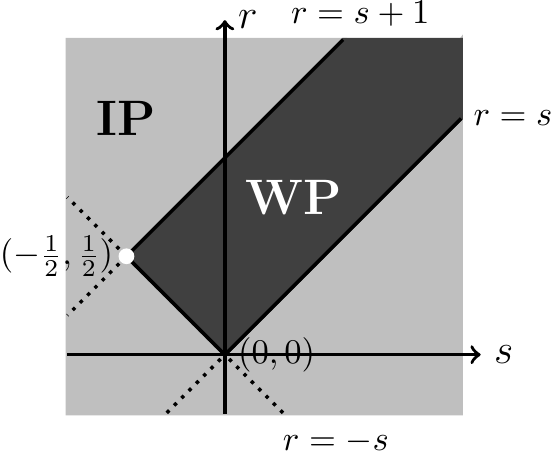}
\caption{Well-posedness and ill-posedness regions}
\label{graph1}
\end{figure}

\subsection{Previous attempts at the critcal point} 
In this subsection, we introduce two partial results 
connected with well-posedness or ill-posedness  
at the critical point \eqref{criticalpoint}. 
In \cite{MNT10}, they proved that the solution map of the Dirac-Klein-Gordon system is 
not twice differentiable at the origin of $H^{-\frac{1}{2}}(\R) \times H^{\frac{1}{2}}(\R)$. 
We remark that the notion of well-posedness does not require twice differentiability for 
the solution map,
and so, this falls slightly short of concluding ill-posedness of the problem.  

Shiota in \cite{S} proved that for any 
initial data $(\psi_0, \phi_0, \phi_1)\in L^{1}(\R)\times L^{\infty}(\R)\times L^1(\R)$, 
there exists a unique solution  $(\psi, \phi)\in L^{1}(\R)\times L^{\infty}(\R)$, and 
the solution map in this setting is continuous. 
Even though Shiota has not published his result, 
there is an introduction to Shiota's result and proof in 
the authors' previous paper \cite{MacOka162}. 
Here we remark the pair of spaces $H^{-\frac12}(\R)$ and $L^1(\R)$, 
and also the pair of $H^{\frac12}(\R)$ and $L^\infty(\R)$ share the 
same scaling property respectively, that is, 
for $f_{\lambda}(x)=f(\lambda x)$, 
\begin{align*}
\frac{\|f_{\lambda}\|_{\dot{H}^{-\frac12}}}{\|f_{\lambda}\|_{L^1}}
=\frac{\|f\|_{\dot{H}^{-\frac12}}}{\|f\|_{L^1}}, \quad
\frac{\|f_{\lambda}\|_{\dot{H}^{\frac12}}}{\|f_{\lambda}\|_{L^{\infty}}}
=\frac{\|f\|_{\dot{H}^{\frac12}}}{\|f\|_{L^{\infty}}}, \qquad
\lambda>0.
\end{align*}
Since there is no inclusion relation between both pairs 
$H^{-\frac12}(\R)$ and $L^1(\R)$, 
and $H^{\frac12}(\R)$ and $L^\infty(\R)$, 
this does not lead to the well-posedness of the problem 
at the critical point.  

\subsection{Main theorem in this paper} 

In this paper, we finish the problem.
We show ill-posedness at the critical point \eqref{criticalpoint} for the problem \eqref{DKG}.
We prove that the solution map of the problem is discontinuous everywhere 
in $H^{-\frac{1}{2}}(\R) \times H^{\frac{1}{2}}(\R) \times H^{-\frac{1}{2}}(\R)$.
Hence this gives a complete classification of 
the range of Sobolev regularity for well-posedness of \eqref{DKG}. Those results 
between the ill-posedness in Sobolev spaces and the well-posedness by 
Shiota \cite{S} in Lebesgue spaces reminds us the similar situations for 
the Chern-Simons-Dirac equation in 1d which is ill-posedness by the current 
authors \cite{MacOka16} in Sobolev spaces and well-posedness by the current first 
author and Ogawa \cite{MacOga17} in Lebesgue spaces. Now we state our main theorem.

\begin{thm} \label{IP}
For any $\psi_0 \in H^{-\frac{1}{2}} (\R)$, $(\phi_0,\phi_1) \in H^{\frac{1}{2}}(\R) \times H^{-\frac{1}{2}}(\R)$, and any $\eps >0$, there exists a solution $(\psi_{\eps},\phi_{\eps})$ to \eqref{DKG} and $t_{\eps} \in (0, \eps )$ such that
\begin{align*}
& \| \psi_{\eps}(0) - \psi_0 \|_{H^{-\frac{1}{2}}} + \| \phi_{\eps} (0) - \phi_0 \|_{H^{\frac{1}{2}}} + \| \partial_t \phi_{\eps}(0) - \phi_1 \|_{H^{-\frac{1}{2}}} < \eps , \\
& \| \phi_{\eps}(t_{\eps}) \|_{H^{\frac{1}{2}}}> \eps^{-1}.
\end{align*}
\end{thm}

As mentioned in \cite{MNT10} and \cite{MacOka162}, 
inconvenient interaction for well-posedness occurs in the 
nonlinearity $\psi ^{\ast} \gamma ^0 \psi$ of the Klein-Gordon equation, 
and, therefore, we expect that the behavior of $\phi$ gives rise
to the ill-posedness. 
The full details of the proof of Theorem \ref{IP} are given in the 
subsequent section, prior to that, we provide an overview	of the 
main ideas. 
We set some sequence of initial data with 
parameter $N$ which we will take the limit $N\to\infty$ later. 
This sequence of initial data converges to 0. 
We apply an iteration argument, and for that,  
we write $\phi$ by the series
\[
\phi = \sum_{k=1}^{\infty} \phi^{(k)},
\]
where $\phi^{(k)}$ is the $k$-th iteration term defined by \eqref{phik} below.
The second iteration term is estimated from below with respect to $N$. 
We show the series converges to the solution of \eqref{DKG} by taking 
the existence time sufficiently small. We show the boundedness 
of $\phi-\phi^{(2)}$ with respect to $N$. 
More precisely, we estimate $\phi^{(2)}$ from below by making use of 
the quadratic interaction of 
linear solutions of the Dirac equation.
We estimate the quadratic interaction, especially of high frequency of 
those solutions, so this is high $\times$ high $\to$ high type 
failure of a bilinear estimate. 
This is different from the abstract 
argument by Bejenaru-Tao \cite{BT}, as similar arguments used by 
Kishimoto-Tsugawa \cite{KT} or the authors' proof for 
the area $s<0$, $r<\frac12, s+r<0$ in \cite{MacOka16} 
where they treated 
high $\times$ high $\to$ low type failure. 
We can't apply such abstract theory to obtain the ill-posedness 
at the critical point \eqref{criticalpoint}. 
Our proof is quite straightforward in some sense, that is, 
an induction argument. We follow the argument by Iwabuchi-Ogawa \cite{IO} (see also \cite{Kis}). 
We estimate each of the iteration terms, and for then, 
we make certain delicate estimates for the second iteration 
term and fortunately, thanks to the smoothing effect of the Duhamel 
terms, it is enough to roughly estimate the higher order iteration 
terms.

\section{Proof of Theorem \ref{IP}}
We will prove Theorem \ref{IP} under the massless case 
$m=M=0$ and norm inflation at zero 
$\psi_0=\phi_0=\phi_1=0$. This is sufficient for the general case in 
Theorem \ref{IP} from the argument in \cite{Oka17}.

\subsection{Preliminaries and iteration setting}
Here we introduce the modulation space $\M (\R)$, see \cite{F}.

\begin{dfn}
Define the space $M^0_{2,1} (\R)$ as the completion of $C_0^{\infty}(\R)$ with respect to the norm
\[
\norm{f}_{\M} = \sum _{k \in \mathbb{Z}} \norm{\wh{f}}_{L^2([k-1, k+1])} .
\]
\end{dfn}

The modulation space $\M (\R)$ satisfies the embeddings
\[
H^{\frac{1}{2}+\eps} (\R) \hookrightarrow \M (\R) \hookrightarrow L^2(\R).
\]
Moreover, $\M (\R)$ is a Banach algebra, in particular, there exists $C_1>1$ such that
\begin{equation} \label{prod}
\| fg \|_{\M} \le C_1 \| f \|_{\M} \| g \|_{\M}
\end{equation}
holds true for any $f,g \in \M (\R)$.

By setting
\[
u := \psi_1-\psi_2, \quad v := \psi_1+\psi_2,
\]
the system \eqref{DKG} with $m=M=0$ is written as follows:
\begin{equation}\label{DKG'}
\begin{cases}
(\partial_{t}+\dx)u= -i \phi v, \quad u(0)=u_{0}, \\
(\partial_{t}-\dx)v= -i \phi u, \quad v(0)=v_{0}, \\
(\partial_{t}^{2}-\dx^{2})\phi = 2 \Re (u\bar{v}), 
\quad \phi(0)=\phi_{0}, \ \partial_{t}\phi(0)=\phi_{1},
\end{cases}
\end{equation}
where $\Re z$ is the real part of $z\in\C$. 
For $N \in{\mathbb{N}}$, we set
\begin{align}\label{initialD1}
& \hat{u}_0(\xi)=\sigma_N\mathbf{1}_{[5N,7N]}(\xi), \\
& \hat{v}_0(\xi)=\sigma_N(\log N)^{-\frac12} \mathbf{1}_{[-N,0]}(\xi), \label{initialD2} \\
& \phi_{0}(x)=\phi_{1}(x)=0,
\end{align}
where $\sigma_{N}$ converges to zero but slower than 
$(\log\log N)^{-\frac12}$ as $N\to\infty$.
Elementary calculations yield
\begin{align}\label{u0}
\|u_{0}\|_{H^{-\frac{1}{2}}}
&=\sigma_N\Big(\int_{5N}^{7N}\frac{1}{\lr{\xi}}d\xi\Big)^{\frac12}
\sim\sigma_N, \\
\|v_{0}\|_{H^{-\frac{1}{2}}}
&=\sigma_N(\log N)^{-\frac12}\Big(\int_{-N}^{0}\frac{1}{\lr{\xi}}d\xi\Big)^{\frac12}
\sim\sigma_N, \label{v0} \\
\|\phi_{0}\|_{H^{\frac{1}{2}}}&=\|\phi_{1}\|_{H^{-\frac{1}{2}}}=0. \label{phi0}
\end{align}
Let us define the first iteration
\[
u^{(1)} (t,x) =u_{0}(x-t), \quad v^{(1)} (t,x) =v_{0}(x+t), \quad \phi^{(1)} (t,x)=0.
\]
For $k \in \Z_{\ge 2}$, we define the higher order iteration functions 
as follows
\begin{align}
u^{(k)} (t,x) &= -i \sum_{\substack{k_1,k_2 \in \N \\ k_1+k_2=k}} \int_0^t (\phi^{(k_{1})}v^{(k_{2})}) \left( t', x- (t-t') \right) dt', \label{uk} \\
v^{(k)} (t,x) &= -i \sum_{\substack{k_1,k_2 \in \N \\ k_1+k_2=k}} \int_0^t (\phi^{(k_{1})}u^{(k_{2})}) \left( t', x+(t-t') \right) dt', \label{vk} \\
\phi^{(k)} (t,x) &= 2 \sum_{\substack{k_1,k_2 \in \N \\ k_1+k_2=k}} \int_0^t \frac{\sin (t-t') |\dx|}{|\dx|} \Re \left( u^{(k_{1})} \ol{v^{(k_{2})}} \right) (t',x) dt'. \label{phik}
\end{align}
We remark here that $\phi^{(1)}=0$ gives
\begin{align*}
u^{(2l)}&=v^{(2l)}=\phi^{(2l+1)}=0
\end{align*}
for any $l \in \mathbb{N}$. We can reduce the number of functions 
which we need to estimate, but it seems that  
this does not help the main part of our argument. It still remains 
to estimate all of the crucial iteration functions.

\subsection{Convergence of iteration terms}
We shall first show that the following expansions converge in  $L^{\infty}([0,T]; \M (\R))$ 
for sufficiently small $T>0$ with fixed $N$,
\begin{align}\label{expand}
u:= \sum_{k=1}^{\infty} u^{(k)}, \quad
v:= \sum_{k=1}^{\infty} v^{(k)}, \quad
\phi := \sum_{k=1}^{\infty} \phi^{(k)},
\end{align}
and, moreover, these limits satisfy \eqref{DKG'}.
We begin by establishing some precise estimates for $u^{(1)}, v^{(1)}$ 
and $\phi^{(2)}$ in details. 
We have a constant $C_2>1$ such that, for any $t>0$,
\begin{align} \label{u1v1}
\begin{split}
\|u^{(1)}(t)\|_{\M} &=\sum_{k\in\Z}\|e^{it\xi}\hat{u}_{0}\|_{L^{2}[k-1,k+1]}
=\sum_{k=5N}^{7N}\|\hat{u}_{0}\|_{L^{2}[k-1,k+1]} \\
&\le C_2\sigma_N N,
\end{split}
\end{align}
and also
\begin{align*}
\begin{split}
\|v^{(1)}(t)\|_{\M} &=\sum_{k\in\Z}\|e^{-it\xi}\hat{v}_{0}\|_{L^{2}[k-1,k+1]}
=\sum_{k=0}^{N}\|\hat{v}_0\|_{L^{2}[k-1,k+1]} \\
&\le C_2\sigma_N(\log N)^{-\frac12}N.
\end{split}
\end{align*}
Since these estimates are uniform with respect to $t$, we also have
\begin{align}\label{u1}
\|u^{(1)}\|_{L^{\infty}_t\M} &\le C_2\sigma_N N, \\
\|v^{(1)}\|_{L^{\infty}_t\M} &\le C_2\sigma_N(\log N)^{-\frac12}N.\label{v1}
\end{align}
We next estimate $\phi^{(2)}$. 
A direct calculation shows
\begin{equation} \label{wtphi}
\begin{aligned}
\wh{\phi^{(2)}}(t,\xi)
= 2\int_0^t \frac{\sin (t-t')\xi}{\xi} \mathcal{F}\left[\Re(u_0(x-t') \ol{v_0(x+t')}) \right] dt'. 
\end{aligned}
\end{equation}
We calculate the Fourier transform. From
 $\wh{\ol{f}}(\xi)=\ol{\wh{f}}(-\xi)$, we have
\begin{align*}
\mathcal{F}\left[u_0(x-t') \ol{v_0(x+t')}\right]&=(\wh{u_0}(\xi)e^{-it'\xi})*(\wh{\ol{v_0}}(\xi)e^{it'\xi}) 
=(\wh{u_0}(\xi)e^{-it'\xi})*(\ol{\wh{v_0}}(-\xi)e^{it'\xi}) \\
&=\int_{\R} \wh{u_0}(\xi-\eta)e^{-it'(\xi-\eta)}\ol{\wh{v_0}}(-\eta)e^{it'\eta}d\eta,
\end{align*}
and so 
\begin{align*}
&2\mathcal{F}\left[\Re(u_0(x-t') \ol{v_0(x+t')}) \right]
=\mathcal{F}\left[u_0(x-t') \ol{v_0(x+t')}+\ol{u_0(x-t')}v_0(x+t')\right] \\
&=\int_{\R} \wh{u_0}(\xi-\eta)e^{-it'(\xi-\eta)}\ol{\wh{v_0}}(-\eta)e^{it'\eta}d\eta
+\ol{\int_{\R} \wh{u_0}(-\xi-\eta)e^{-it'(-\xi-\eta)}\ol{\wh{v_0}}(-\eta)e^{it'\eta}d\eta} \\
&=\int_{\R} \wh{u_0}(\xi-\eta)e^{-it'(\xi-\eta)}\ol{\wh{v_0}}(-\eta)e^{it'\eta}d\eta
+\int_{\R} \ol{\wh{u_0}}(-\xi-\eta)e^{-it'(\xi+\eta)}\wh{v_0}(-\eta)e^{-it'\eta}d\eta. 
\end{align*}
Substituting into \eqref{wtphi}, we get two terms, the first of which is 
\begin{align}\label{wtphi1}
\begin{split}
&\int_0^t \frac{\sin (t-t')\xi}{\xi} 
\int_{\R} \wh{u_0}(\xi-\eta)e^{-it'(\xi-\eta)}\ol{\wh{v_0}}(-\eta)e^{it'\eta}
d\eta dt' \\
&=\int_{\R} \wh{u_0}(\xi-\eta)\ol{\wh{v_0}}(-\eta)
\int_0^t\frac{e^{it\xi}e^{2it'(-\xi+\eta)}-e^{-it\xi}e^{2it'\eta}}{2i\xi}
dt'd\eta \\
&=\frac{1}{4\xi}\int_{\R} \left(e^{it\xi}\frac{e^{2it(\eta-\xi)}-1}{\xi-\eta}+e^{-it\xi}\frac{e^{2it\eta}-1}{\eta}\right)
\wh{u_0}(\xi-\eta)\ol{\wh{v_0}}(-\eta)d\eta,
\end{split}
\end{align}
and the second of which is 
\begin{align}\label{wtphi2}
\begin{split}
&\int_0^t \frac{\sin (t-t')\xi}{\xi} 
\int_{\R} \ol{\wh{u_0}}(-\xi-\eta)e^{-it'(\xi+\eta)}{\wh{v_0}}(-\eta)e^{-it'\eta}
d\eta dt' \\
&=\frac{1}{4\xi}\int_{\R} \left(e^{it\xi}\frac{e^{-2it(\xi+\eta)}-1}{\xi+\eta}-e^{-it\xi}\frac{e^{-2it\eta}-1}{\eta}\right)
\ol{\wh{u_0}}(-\xi-\eta){\wh{v_0}}(-\eta)d\eta.
\end{split}
\end{align}
From here we consider these terms with the sequence of initial data \eqref{initialD1} 
and \eqref{initialD2} for $\hat{u}_0$ and $\hat{v}_0$ respectively. 
Since these are real-valued functions, 
the complex conjugate disappear. 
The first term on the right-hand side of \eqref{wtphi1} is estimated as follows:
\begin{equation} \label{phi2-1}
\left|\frac{1}{\xi}\int_{\R} e^{it\xi} \frac{e^{2it(\eta-\xi)}-1}{\xi-\eta} \wh{u_0}(\xi-\eta) \wh{v_0}(-\eta) d\eta \right|
\le\sigma_N^2(\log N)^{-\frac{1}{2}}\frac{\bm{1}_{[5N,8N]}(\xi)}{\xi}.
\end{equation}
For the second term in \eqref{wtphi1}, 
we change the variable $\eta$ by $\eta/t$ and estimate
\begin{align}\label{phi2-2}
\begin{split}
&\left|\frac{1}{\xi}\int_{\R} e^{-it\xi} \frac{e^{2it\eta}-1}{\eta} \wh{u_0}(\xi-\eta) \wh{v_0}(-\eta) d\eta \right| \\
&\le \sigma_N^2  (\log N)^{-\frac{1}{2}}\int_0^{tN} \left| \frac{e^{2i\eta}-1}{\eta} \right| d\eta \ 
\frac{\bm{1}_{[5N,8N]}(\xi)}{\xi} \\
&\le\sigma_N^2  (\log N)^{-\frac{1}{2}}( 2 + \log tN) \ 
\frac{\bm{1}_{[5N,8N]}(\xi)}{\xi}
\end{split}
\end{align}
for $tN \ge 1$. 
The estimate for \eqref{wtphi2} is similar but the support of the corresponding characteristic 
function is in $[-8N, -5N]$. 
In total, we obtain
\[
|\wh{\phi^{(2)}}(t,\xi)|
\le 2\sigma_N^2(\log tN) (\log N)^{-\frac{1}{2}}\frac{\bm{1}_{[5N,8N]}(\xi)+\bm{1}_{[-8N,-5N]}(\xi)}{\xi},
\]
provided that $tN \gg 1$. 
Therefore, from
\begin{align*}
\left\|\mathcal{F}^{-1}\left(\frac{\bm{1}_{[5N,8N]}(\xi)+\bm{1}_{[-8N,-5N]}(\xi)}{\xi}\right)\right\|_{\M}
\sim\frac{8N-5N}{N}\sim 1,
\end{align*}
we have
\begin{align}\label{phi2}
\|\phi^{(2)}\|_{L^{\infty}_t\M} &\le C_{3}\sigma_N^2(\log tN)(\log N)^{-\frac12}.
\end{align}
We now set the time $t=t(N)$ as follows
\[
t= N^{-1} (\log N)^{(\log N)^{\frac{1}{2}}}.
\]
We enumerate properties of the terms 
for large $N$ which we will use below: 
As $N\to\infty$, then $t\to0$ and 
\begin{align}
tN&= (\log N)^{(\log N)^{\frac{1}{2}}}\to\infty, \\
\log (tN)&=(\log N)^{\frac{1}{2}}\log\log N\to\infty, \\
t^{3}N^{2}&=N^{-1}(\log N)^{3(\log N)^{\frac{1}{2}}}\to0, \label{N-oder3}\\
t^{3}N^{\frac52}&=N^{-\frac12}(\log N)^{3(\log N)^{\frac{1}{2}}}\to0. \label{N-oder4}
\end{align}

To estimate the higher order iteration terms, we use the following lemma (\cite{Kis}, see also Lemma 4.2 in \cite{MacOka16}):
\begin{lem} \label{lem:Kis}
Let $\{ a_n \}$ be a positive sequence.
Assume
\begin{equation} \label{seq}
a_n \le C \sum _{\substack{n_1, n_2 \in \mathbb{N}, \\ n_1+n_2=n}} a_{n_1} a_{n_2}
\end{equation}
holds.
Then, we have
\[
a_n \le \left( \frac{2}{3} \pi ^2 C \right)^{n-1} a_1^n.
\]
\end{lem}

\begin{lem} \label{lem:iteration}
There exists a constant $C>1$ such that for any $l \in \N$ and $t \gg N^{-1}$ the following estimates hold:
\begin{align*}
\| \phi^{(2l)}\|_{L^{\infty}_t\M} 
&\le\sigma_{N}^{2l} (\log tN) (\log N)^{-\frac{1}{2}} 
\left(Ct^3N^2 \right)^{l-1}, \\
\| u^{(2l+1)}\|_{L^{\infty}_t\M} 
&\le\sigma_{N}^{2l+1} (\log tN) (\log N)^{-1} tN 
\left(Ct^3N^2 \right)^{l-1}, \\
\| v^{(2l+1)}\|_{L^{\infty}_t\M} 
&\le\sigma_{N}^{2l+1}(\log tN) (\log N)^{-\frac{1}{2}} tN
\left(Ct^3N^2 \right)^{l-1}.
\end{align*}
\end{lem}

\begin{proof}
Let $\{ a_k \}$ be the sequence defined by
\[
a_1= \max (C_2,C_3) , \quad a_k := 2C_1 \sum_{\substack{k_1,k_2 \in \N \\ k_1+k_2=k}} a_{k_1} a_{k_2}
\]
for $k \ge 2$, where $C_1, C_{2}$ and $C_3$ are the constants appearing in \eqref{prod}, \eqref{u1v1} and \eqref{phi2} respectively.
From Lemma \ref{lem:Kis}, it suffices to show that
\begin{align}
\| \phi^{(2l)}\|_{L^{\infty}_t\M} 
&\le a_{2l}\sigma_{N}^{2l} (\log tN) (\log N)^{-\frac{1}{2}} \left( t^3N^2 \right)^{l-1}, \label{phi2l} \\
\| u^{(2l+1)}\|_{L^{\infty}_t\M} 
&\le a_{2l+1}\sigma_{N}^{2l+1}(\log tN) (\log N)^{-1} tN \left( t^3N^2 \right)^{l-1}, \label{u2l+1} \\
\| v^{(2l+1)}\|_{L^{\infty}_t\M} 
&\le a_{2l+1}\sigma_{N}^{2l+1}(\log tN) (\log N)^{-\frac{1}{2}} tN \left( t^3N^2 \right)^{l-1}. \label{v2l+1}
\end{align}
We use an induction argument to obtain these estimates.
We have done the estimate for $\phi^{(2)}$ in \eqref{phi2} 
which is \eqref{phi2l} with $l=1$.
By \eqref{phik}, we have 
\begin{align*}
\|\phi^{(2l)}(t)\|_{\M}&\le 2 \sum_{\substack{k_1,k_2 \in \N \\ k_1+k_2=2l}} 
\int_0^t(t-t')\|u^{(k_{1})} \ol{v^{(k_{2})}} (t')\|_{\M} dt' \\
&\le\sum_{\substack{k_1,k_2 \in \N \\ k_1+k_2=2l}}C_{1}t^{2}
\|u^{(k_{1})}\|_{L^{\infty}_{t}\M}\|{v^{(k_{2})}}\|_{L^{\infty}_{t}\M}.
\end{align*}
Therefore
\begin{align}\label{induct-phi}
\|\phi^{(2l)}\|_{L^{\infty}_{t}\M}
&\le C_{1}t^{2}\sum_{\substack{k_1,k_2 \in \N \\ k_1+k_2=2l}}
\|u^{(k_{1})}\|_{L^{\infty}_{t}\M}\|{v^{(k_{2})}}\|_{L^{\infty}_{t}\M}.
\end{align}
Similarly, by \eqref{uk} and \eqref{vk}, we have
\begin{align}\label{induct-u}
\|u^{(2l+1)}\|_{L^{\infty}_{t}\M}
&\le C_{1}t\sum_{\substack{k_1,k_2 \in \N \\ k_1+k_2=2l+1}}
\|\phi^{(k_{1})}\|_{L^{\infty}_{t}\M}\|{v^{(k_{2})}}\|_{L^{\infty}_{t}\M}, \\
\|v^{(2l+1)}\|_{L^{\infty}_{t}\M}
&\le C_{1}t\sum_{\substack{k_1,k_2 \in \N \\ k_1+k_2=2l+1}}
\|\phi^{(k_{1})}\|_{L^{\infty}_{t}\M}\|{v^{(k_{2})}}\|_{L^{\infty}_{t}\M}.\label{induct-v}
\end{align}
We apply an induction argument with \eqref{u1}, \eqref{v1}, 
\eqref{phi2l}, \eqref{u2l+1} and \eqref{v2l+1}. 
So we treat with $u^{(1)}, v^{(1)}$ and others differently. 
Suppose that the estimates \eqref{phi2l}--\eqref{v2l+1} hold up to some $l \in \N$.
Then, from \eqref{induct-phi}, we have for large $N$, 
\begin{align*}
&\|\phi^{(2l+2)}\|_{L^{\infty}\M} \\
&\le C_{1}t^{2}\Bigg\{\|u^{(2l+1)}\|_{L^{\infty}_{t}\M}\|{v^{(1)}}\|_{L^{\infty}\M}
+\|u^{(1)}\|_{L^{\infty}_{t}\M}\|v^{(2l+1)}\|_{L^{\infty}_{t}\M} \\
& \hspace*{80pt} + \sum_{\substack{l_1,l_2 \in \N \\ l_1+l_2=l}}
\|u^{(2l_1+1)}\|_{L^{\infty}_{t}\M}\|v^{(2l_2+1)}\|_{L^{\infty}\M}\Bigg\}\\
&\le C_{1}a_{2l+1}a_1\sigma_{N}^{2l+2}(\log tN)(\log N)^{-\frac{3}{2}}(t^3N^2)^l \\
&\qquad+ C_{1}a_{1}a_{2l+1}\sigma_{N}^{2l+2}(\log tN)(\log N)^{-\frac{1}{2}}(t^3N^2)^l \\
& \hspace*{30pt} +  C_{1}\sum_{\substack{l_1, l_2 \in \N \\ l_1+l_2=l}}
a_{2l_1+1}a_{2l_2+1}\sigma_{N}^{2l+2} (\log tN)^2 (\log N)^{-\frac{3}{2}} (tN)^{-2} (t^3N^2)^l \\
&\le a_{2l+2}\sigma_{N}^{2l+2}(\log tN)(\log N)^{-\frac{1}{2}}(t^3N^2)^l.
\end{align*}
Similarly, for large $N$, we use \eqref{induct-u}
\begin{align*}
&\| u^{(2l+3)}\|_{L^{\infty}_{t}\M} \\
&\le C_{1}t \Bigg\{ \| \phi^{(2l+2)}\|_{L^{\infty}_{t}\M}
\|v^{(1)}\|_{L^{\infty}_{t}\M}
+\sum_{\substack{l_1,l_2 \in \N \\ l_1+l_2=l+1}} 
\|\phi^{(2l_1)}\|_{L^{\infty}_{t}\M}\|v^{(2l_2+1)}\|_{\M} \Bigg\}\\
&\le C_1 a_{2l+2}a_1\sigma_{N}^{2l+3}(\log tN)(\log N)^{-1}tN(t^3N^2)^l \\
& \hspace*{30pt} + C_1 \sum_{\substack{l_1, l_2 \in \N \\ l_1+l_2=l+1}} a_{2l_1} a_{2l_2+1}\sigma_{N}^{2l+3}(\log tN)^2 (\log N)^{-1} (tN)^{-1} (t^3N^2)^l \\
&\le a_{2l+3}\sigma_{N}^{2l+3}
(\log tN) (\log N)^{-1} tN (t^3N^2)^l,
\end{align*}
and use \eqref{induct-v} to have
\begin{align*}
&\| v^{(2l+3)}\|_{L^{\infty}_{t}\M} \\
&\le C_{1}t \Bigg\{ \| \phi^{(2l+2)}\|_{L^{\infty}_{t}\M}
\|u^{(1)}\|_{L^{\infty}_{t}\M} 
+ \sum_{\substack{l_1,l_2 \in \N \\ l_1+l_2=l+1}} 
\| \phi^{(2l_1)}\|_{L^{\infty}_{t}\M}
\|u^{(2l_2+1)}\|_{L^{\infty}_{t}\M} \Bigg\} \\
&\le C_1 a_{2l+2} a_1 \sigma_{N}^{2l+3}
(\log tN) (\log N)^{-\frac{1}{2}} tN (t^3N^2)^l \\
& \hspace*{30pt} + C_1 \sum_{\substack{l_1, l_2 \in \N \\ l_1+l_2=l+1}} a_{2l_1} a_{2l_2+1}\sigma_{N}^{2l+3}
(\log tN)^2 (\log N)^{-\frac{3}{2}} (tN)^{-1} (t^3N^2)^l \\
&\le a_{2l+3}\sigma_{N}^{2l+3}
(\log tN) (\log N)^{-\frac{1}{2}} tN (t^3N^2)^l.
\end{align*}
Therefore, the estimates \eqref{phi2l}--\eqref{v2l+1} hold true.
\end{proof}

This lemma says that the series \eqref{expand} converges in 
$L^{\infty}([0,T]; \M (\R))$ provided that $T^3N^2<1$ which is 
satisfied for large $N$ from \eqref{N-oder3}.
Moreover, since we have a condition on the support of the 
iteration functions, we estimate the Sobolev norm with respect to 
the $x$ variable as follows
\begin{align*}
\begin{split}
\|\phi^{(2l)}(t)\|_{H^{\frac{1}{2}}}
&=\| \mathcal{F}^{-1} [\bm{1}_{[-100^l N, 100^l N]}] \ast \phi^{(2l)}(t) \|_{H^{\frac{1}{2}}} \\
&\lesssim 100^l N^{\frac{1}{2}}  \| \phi^{(2l)}(t) \|_{\M} \\
&\lesssim\sigma_{N}^{2l}(\log tN)(\log N)^{-1/2}N^{\frac12}
(100 C t^{3}N^{2})^{l-1}.
\end{split}
\end{align*}
Then the same limit $\phi = \sum_{k=1}^{\infty} \phi^{(k)}$ in the 
modulation space $\M$ as above
also exists in $L^{\infty}([0,T];H^{\frac{1}{2}}(\R))$ 
provided that $100 C T^3N^2<1$.
Moreover if we take 
$100 C T^{3}N^{\frac52}<1$ we will have an extra $N^{-\frac{l-1}2}$ factor  
and estimate for each $l\ge2$  
\begin{align*}
\|\phi^{(2l)}(t)\|_{H^{\frac{1}{2}}}
&\lesssim\sigma_{N}^{2l}(\log tN)(\log N)^{-1/2}N^{\frac12}N^{-\frac{l-1}2}
(100 C t^{3}N^{\frac52})^{l-1} \\
&\lesssim\sigma_{N}^{2l}(\log tN)(\log N)^{-1/2}
(100 C t^{3}N^{\frac52})^{l-1}
\end{align*}
where $l=2$ was the worst case but it holds, 
and other cases $l\ge3$ were easier. 
Under the condition $100 C T^{3}N^{\frac52}<1$ which holds 
from \eqref{N-oder4}, we have
\begin{align}\label{phi2ls}
\sum_{l=2}^{\infty}\|\phi^{(2l)}\|_{L^{\infty}_{T}H^{\frac12}}
\lesssim \sigma_N^4 (\log tN)(\log N)^{-1/2}. 
\end{align}

\subsection{Lower bound of $\phi^{(2)}$ and conclusion}
Here, we establish an appropriate lower bound 
for $\| \phi^{(2)} (t) \|_{H^{\frac{1}{2}}}$.
We decomposed $\widehat{\phi^{(2)}}$ in \eqref{wtphi} into three terms, 
\eqref{wtphi2}, \eqref{phi2-1} and \eqref{phi2-2}. 
It suffices to establish a lower bound on \eqref{phi2-2} only 
since \eqref{wtphi2} is
negligible if we restrict $\xi\ge0$ in the norm and we have seen that 
\eqref{phi2-1} converges to zero faster than \eqref{phi2-2}.  
We write \eqref{phi2-2} here again and estimate
\begin{align*}
&\left|\frac{1}{\xi}\int_{\R} e^{-it\xi} \frac{e^{2it\eta}-1}{\eta} \wh{u_0}(\xi-\eta) \wh{v_0}(-\eta) d\eta\right| \\
&\gtrsim\sigma_{N}^{2}(\log N)^{-1/2}
\left|\int_{0}^{N}\frac{e^{2it\eta}-1}{\eta}d\eta\right|
\frac{\bm{1}_{[6N,7N]}(\xi)}{|\xi|} \\
&\gtrsim\sigma_{N}^{2}(\log N)^{-1/2}
\log(tN)N^{-1}\bm{1}_{[6N,7N]}(\xi).
\end{align*}
We obtain
\begin{align*}
\|\phi^{(2)}(t)\|_{H^{\frac{1}{2}}}
&\ge\left(\int_{0}^{\infty}\lr\xi|\widehat{\phi^{(2)}}(t,\xi)|^{2}d\xi\right)^{1/2} \\
&\gtrsim\sigma_{N}^{2}(\log N)^{-1/2}\log(tN)N^{-1}
\left(N(7N-6N)\right)^{1/2} \\
&\sim\sigma_{N}^{2}(\log N)^{-\frac{1}{2}}(\log tN).
\end{align*}
Therefore, the triangle inequality and Lemma \ref{lem:iteration} with \eqref{phi2ls} yield
\begin{align}\label{phi-norm}
\begin{split}
\| \phi(t) \|_{H^{\frac{1}{2}}}
&\ge \|\phi^{(2)}(t)\|_{H^{\frac{1}{2}}} - \sum_{l=2}^{\infty} \|\phi^{(2l)}(t)\|_{H^{\frac{1}{2}}} \\
&\gtrsim\sigma_{N}^{2}(\log N)^{-\frac{1}{2}}(\log tN)
-\sigma_N^4 (\log N)^{-\frac{1}{2}}(\log tN) \\
&\gtrsim\sigma_{N}^{2}(\log N)^{-\frac{1}{2}}(\log tN)
=\sigma_{N}^{2}\log\log N.
\end{split}
\end{align}
Since $\sigma_{N}$ converges to zero slower than $(\log\log N)^{-\frac12}$ 
as $N\to\infty$, 
the initial data \eqref{u0}, \eqref{v0} and \eqref{phi0} converge 
to zero, still the solution \eqref{phi-norm} is bounded from below. 
Therefore we conclude the norm inflation for \eqref{DKG'}.

\section*{Acknowledgment}
The first and second authors were supported by 
JSPS KAKENHI Grant number JP16K05191 and JP16K17624 respectively.

\end{document}